\documentclass[11pt,twoside]{amsart}
\usepackage{pict2e,amsmath,amsthm,amsfonts,amssymb}
\usepackage[all]{xy}
\usepackage{xcolor}
\usepackage[pdftex]{hyperref}
\hypersetup{colorlinks,citecolor=green!70!black,linkcolor=red!80!black,backref=page}

\newtheorem{theorem}{Theorem}[section]
\newtheorem{definition}[theorem]{Definition}

\newtheorem{claim}[theorem]{Claim} \newtheorem{lemma}[theorem]{Lemma}
\newtheorem{corollary}[theorem]{Corollary}
\newtheorem{proposition}[theorem]{Proposition}
\newtheorem{question}[theorem]{Question}

\newcommand{\bP}{\mathbb{P}}
\newcommand{\bR}{\mathbb{R}}
\newcommand{\cL}{\mathcal{L}}
\newcommand{\cO}{\mathcal{O}}

\newcommand{\Q}{\mathbb{Q}} \newcommand{\GG}{\mathbb{G}}
\newcommand{\x}{{\mathbf{x}}}
\newcommand{\Pro}{\mathbb{P}}
\newcommand{\Gr}{\mathbf{Gr}} \newcommand{\dblq}{/\!/}
\newcommand{\M}{\overline{M}} \newcommand{\oM}{\overline{M}} 
 
\newcommand{\Nef}{\operatorname{Nef}}

\newcommand{\GC}{\mathbf{GC}}
\newcommand{\mtwo}{\texttt{M2}}

\newcommand{\SGC}{\mathbf{SGC}} 
\newcommand{\Pic}{\operatorname{Pic}}
\newcommand{\V}{V}

\begin{document}

\pagenumbering{arabic}
\title{Nef divisors on $\overline{M}_{0,n}$ from GIT}
\author{Valery Alexeev and David Swinarski}
\date{\today}

\maketitle

\begin{abstract}
  We introduce and study the GIT Cone of $\overline{M}_{0,n}$, which is
  generated by the pullbacks of the natural ample line bundles on the
  GIT quotients $(\mathbb P^1)^n//SL(2)$. We give an explicit formula for these
  line bundles and prove a number of basic results about the GIT
  cone. 

  As one application, we prove unconditionally that the log canonical
  models of $\overline{M}_{0,n}$ with a symmetric boundary divisor coincide
  with the moduli spaces of weighted curves or with the symmetric GIT
  quotient, extending the result of Matt Simpson
\end{abstract}

\section{Introduction}

The moduli space of smooth pointed genus zero curves $M_{0,n}$ has
many compactifications.  Among the most important of these is
$\M_{0,n}$, the moduli space of Deligne-Mumford stable curves.  There
are many beautiful results on these spaces; for instance, Keel and
Kapranov described them as explicit blowups of $(\bP^1)^{n-3}$ and
$\mathbb P^{n-3}$ (\cite{Keel,Kap1}), and Kapranov identifed
$\M_{0,n}$ with the Chow quotient of the Grassmannian $\Gr(2,n)$ by a
torus.  But the birational geometry of these spaces is still not fully
understood, and it is known to be very complicated. For example, Faber
computed that the nef cone of $\M_{0,6}$ has 3190 extremal rays.

Another family of compactifications of $M_{0,n}$ is provided by the
GIT quotients $(\Pro^{1})^{n} \dblq_{\x} SL(2)$, where the vector $\x$
specifies the linearization 
\linebreak
$\mathcal{O}(x_{1},\ldots,x_{n})$.  Let
$\mathcal{L}_{\x}$ denote the distinguished polarization on the GIT
quotient descending from the linearization. 

By \cite{Kap2}, there are birational morphisms $\pi_{\x}: \M_{0,n} \rightarrow
(\Pro^{1})^{n} \dblq_{\x} SL(2)$, and hence the pullbacks
$L_{x}:=\pi_{\x}^{*} \mathcal{L}_{\x} $ of the distinguished
polarizations on the GIT quotients are big and nef, since $\pi_{\x}$ is
birational and 
$\mathcal{L}_{\x}$ is ample.  In this paper we consider the subcone
generated by the $L_{\x}$ inside the nef cone of $\M_{0,n}$.  We call
this the \textit{GIT cone} and denote it $\GC$.

Our first main result Lemma~\ref{phi lemma} is an
explicit formula for the intersections of line bundles $L_{\x}$ with
the $F$-curves in $\M_{0,n}$.  
This allows us to describe a finite set
of generators for the GIT cone, which we call the \textit{GIT complex
  0-cells}.  We show that for $n\ge 6$ $\GC$ is strictly smaller than
the nef cone.
 
The symmetric group $S_{n}$ acts on $\M_{0,n}$ by permuting the
points, and symmetric divisors often play a key role.  Let $S_{\x}$
denote the symmetrization of $L_{\x}$, that is, $S_{\x} :=
\bigotimes_{\sigma \in S_{n}} L_{\sigma \x} $.  We call the cone generated
by these the \textit{symmetrized GIT cone}, and denote it $\SGC$.  We
give a conjectural list of extremal rays of $\SGC$ and give an example
showing that $\SGC$ is strictly smaller than the symmetric nef cone.

Thus, the $L_{\x}$ and $S_{\x}$ do not completely describe the Mori
theory of $\M_{0,n}$.  Nevertheless, they contain many geometrically
interesting divisors, including the classes $\psi_i$ and $\kappa\equiv
K+\Delta$, 
(see Sections~\ref{sec:psi-is-GIT} and \ref{rewriteAalpha})
and we believe these cones are very useful for the study of $\M_{0,n}$.
One major advantage is a simple and explicit way of computing them
that we give. 
As an illustration, we reprove a result of Simpson and
Fedorchuk--Smyth, also recently proved by Kiem--Moon (after the
preprint version of this paper appeared).

Matthew Simpson proved the following results in his dissertation: Assume
the $S_n$-equivariant Fulton Conjecture (which is known for $n \leq
24$ \cite{Gib}).  If certain divisors $A_{\alpha}$ (defined below) on
$\M_{0,n}$ are nef, then the log canonical models of the pair
$(\M_{0,n}, \Delta)$ can be identified either as $\M_{0,\beta}$ for
symmetric sets of weights $\beta$ or as $(\Pro^{1})^{n} \dblq SL(2)$
with the symmetric linearization (cf. \cite{SimpA} Corollary 2.3.5,
Theorem 2.4.5.)  Simpson then proves that these divisors $A_{\alpha}$
are F-nef (\cite{SimpA} Proposition 2.4.6).  Using the
$S_{n}$-equivariant Fulton conjecture again, this yields as a
corollary that the log canonical models are indeed the spaces claimed.

In Section \ref{rewriteAalpha}, we express Simpson's divisors $A_{\alpha}$ as explicit effective combinations of a small number of the $L_{\x}$.  In particular, this
proves that Simpson's $A_{\alpha}$ are nef without using the
$S_{n}$-equivariant Fulton Conjecture.  This removes one of his two
uses of the $S_{n}$-equivariant Fulton Conjecture.  In Section
\ref{amplesection}, we prove that $A_{\alpha}$ is ample on the
appropriate $\M_{0,\beta}$, removing Simpson's second use of the
$S_{n}$-equivariant Fulton Conjecture.

\textit{Remark.}  Fedorchuk and Smyth (\cite{FS}) and Kiem and Moon (\cite{KM10}) have given proofs of Simpson's result which are also fully independent of the
$S_{n}$-equivariant Fulton Conjecture.  

\subsection*{Acknowledgements} We would like to thank Angela Gibney
and David Smyth for helpful discussions regarding this work.  We are also grateful to Boris Alexeev, who wrote a C++ program for finding the 0-cells of the GIT complex discussed below.

\section{The main line bundles and their intersection theory} \label{mainlemmasection}

Let $\x$ be a vector in $\Q^{n}$ such that $0 < x_{i} \leq 1$ for all
$i$ and $\sum_{i=1}^{n} x_{i} = 2$.  We associate a line bundle $\mathcal{L}_{\x}$ to $\x$ in the following way: Let $\ell$ be the least common multiple of the denominators of the $x_{i}$, and write $\tilde{x}_{i} = \ell x_{i}$.  Since the group $SL(2)$ has no nontrivial characters, the line bundle $\mathcal{O}(\tilde{x}_{1},\ldots,\tilde{x}_n)$ has a unique linearization, and the GIT quotient $(\Pro^{1})^{n} \dblq_{\x} SL(2)$ has a distinguished ample line bundle $\mathcal{L}_{\x}$.   On the other hand, $\x$ gives a linearization
for the action of $T=\mathbb G_m^n/\operatorname{diag} \mathbb G_m$ on
$\Gr(2,n)$, and by \cite{Kap2} one has the Gelfand-MacPherson
correspondence $(\Pro^{1})^{n} \dblq_{\x} SL(2) \cong \Gr(2,n)\dblq_{\x} T$
(see \cite{HMSV} for a nice exposition).

There is a birational morphism $\pi_{\x}: \M_{0,n}
\rightarrow (\Pro^{1})^{n} \dblq_{\x} SL(2)$ (\cite{Kap1},
\cite{Kap2}, \cite{Hass}).

\begin{definition} $ L_{\x} := \pi_{\x}^{*} \mathcal{L}_{\x}.$
\end{definition}

The line bundle $L_{\x}$ is nef, since it is the pullback of an ample
bundle under a morphism. If all $x_i<1$ then the morphism $\pi_{\x}$
is birational, and then $L_{\x}$ is also big. 

We can extend the definition to cover the case when some coordinates
of $\x$ are zero.  Start with $n=4$.  Then if any $x_i =0$, $L_{\x}$
is the trivial line bundle.  If $n \geq 4$ and $x_{i} =0$, then define
$L_{\x} = \pi_{i}^{*} L_{\hat{\mathbf{x}}}$, where $\pi_{i}: \M_{0,n}
\rightarrow \M_{0,n-1}$ is the morphism which forgets the $i^{th}$
marked point and $\hat{\mathbf{x}}$ is the vector obtained from $\x$
by omitting the $i^{th}$ coordinate. We also formally set $L_{\x}=0$
if any $x_i=1$.

The following lemma allows us to compute intersections of the $L_{\x}$
with $F$-curves:

\begin{lemma} \label{phi lemma} Let $\{ a,b,c,d \}$ be a partition of
the set $\{1,\ldots,n\}$ into four nonempty subsets, and let
$F_{a,b,c,d}$ be the corresponding $F$-curve class.  Let $\x$ be a
set of weights as above.  Write $x_{a} = \sum_{i \in a} x_{i}$, $x_{b}
= \sum_{i \in b} x_{i}$, etc.  We abbreviate
\begin{displaymath} \min  :=  \min \{ x_{a}, x_{b}, x_{c}, x_{d} \},
\qquad \max  :=  \max \{ x_{a}, x_{b}, x_{c}, x_{d}
\}. 
\end{displaymath} Then
\begin{equation} L_{\x} \cdot F_{a,b,c,d}= \left\{
\begin{array}{ll} 0 & \mbox{ if $\max \geq 1$,} \\ 2(1-\max) & \mbox{
if $\max \leq 1$ and $\max+\min \geq 1$,}\\ 2 \min & \mbox{ if $\max
\leq 1$ and $\max + \min \leq 1$.}
\end{array} \right.
\end{equation}
\end{lemma}

\begin{proof} We reduce the calculation to $\M_{0,4}$ and exploit the
  fact that $\M_{0,4} \cong \Pro^{1}$ is in a natural way a conic in
  the toric variety $\bP^2$.  For simplicity of notation, we assume
  that the partition $\{ a,b,c,d \}$ is ordered, i.e. $a =
  \{1,\ldots,\#a\}$, $ b= \{ \#a+1,\ldots, \#a+\#b\}$, etc.

Recall (\cite{Keel},\cite{KM}) that an $F$-curve in $\M_{0,n}$
parametrizes nodal curves whose dual graph is a trivalent tree except
at one vertex, which is 4-valent.  As the cross ratio of the four
nodes on the corresponding component varies, a $\Pro^{1}$ in
$\M_{0,n}$ is swept out.  The 4-valent vertex partitions in the $n$
leaves into 4 groups, and this is recorded by a partition
$\{ a,b,c,d \}$.  Two $F$-curves are linearly equivalent if and only if
they have the same partition, so the partition specifies a class
$F_{a,b,c,d}$.

Now consider the map $(\bP^1)^4\to (\bP^1)^n$ which is the product of
the diagonal map $\Delta_a:\bP^1\to (\bP^1)^a$ and the similar maps
$\Delta_b,\Delta_c,\Delta_d$. The pullback of the linearization $\x$
on $(\bP^1)^n$ to $(\bP^1)^4$ is $\x'=(x_a,x_b,x_c,x_d)$. Therefore,
the pullback of $\mathcal L_{\x}$ to $(\Pro^{1})^{4} \dblq_{\x'}
SL(2)$ is $\mathcal L_{\x'}$. This gives the following commutative diagram:
\begin{displaymath} 
\xymatrix{ 
\M_{0,n} \ar[rr]^{\pi_{\x}} &&
(\Pro^{1})^{n} \dblq_{\x} SL(2) \\ 
\M_{0,4} \ar[u]^{i} \ar[rr]^{\cong} &&
(\Pro^{1})^{4} \dblq_{\x'} SL(2) \ar[u]^{i_{a,b,c,d}}
}
\end{displaymath}
in which the lower arrow is an isomorphism because both varieties are
isomorphic to $\bP^1$. The class of $i_*[\M_{0,4}]$ in $\M_{0,n}$ is
$F_{a,b,c,d}$. Thus, to complete the proof we just need to compute the
degree of the sheaf $\cL_{\x'}$ on $\bP^1$. 
By the Gelfand-MacPherson correspondence 
$$(\Pro^{1})^{4}
\dblq_{\x'} SL(2) \cong \Gr(2,4)
\dblq_{\x'}, T$$ where $T = \GG_{m}^{4} / \mbox{diag} \, 
\GG_{m}$.  The Pl\"{u}cker embedding descends to a map of torus
quotients
\[ \Gr(2,4) \dblq_{\x'} T \hookrightarrow \Pro^{5} \dblq_{\x'} T
\] 
The GIT quotient $\Pro^{5} \dblq_{\x'} T$, together with the polarization
given by the GIT construction, is the polarized toric variety
corresponding to the fiber over the point $\x'=(x_a,x_b,x_c,x_d)$ of the
polytopal map from the simplex $\sigma_6\subset \bR^6$ with 6 vertices to the
hypersimplex $\Delta(2,4)$ to $\bR^4$. 

Denote the coordinates in $\bR^6$ by $z_{ij}=z_{ji}$ with $i\ne
j$, and the coordinates in $\bR^4$ by
$x_i$. Then the map is given by $x_i=\sum_{j\ne i}z_{ij}$. 
Assume that $x_a\ge x_b\ge x_c\ge x_d$ and that $x_a+x_b\ge
x_c+x_d$. Then an easy explicit computation shows that if $x_a>1$ then
the fiber over $\x'$ is empty, and if $x_a\le1$ then it is the
triangle with the vertices  
\begin{eqnarray*}
\dfrac12(2x_2, {x_1-x_2+x_3-x_4},{x_1-x_2-x_3+x_4}, 0,0,{-x_1+x_2+x_3+x_4})\\
\dfrac12({x_1+x_2-x_3-x_4},2x_3, {x_1-x_2-x_3+x_4}, 0,{-x_1+x_2+x_3+x_4}, 0)\\
\dfrac12({x_1+x_2-x_3-x_4},{x_1-x_2+x_3-x_4}, 2x_4,{-x_1+x_2+x_3+x_4}, 0,0)
\end{eqnarray*}
This polytope is a standard triangle, shifted and dilated by a factor
$\lambda = 1-x_a = 1-\max$. Thus, $(\bP^5,\cO(1))\dblq_{\x'}T =
(\bP^2,\cO(\lambda))$. 
The quotient $\Gr(2,4)\dblq T$ is a conic in this $\bP^2$ and 
$\mathcal L_{\x'}=\mathcal O_{\bP^1}(2\lambda)$. 

This proves the formula in the first two cases. The third case holds
by symmetry.
\end{proof}

\textit{Remark.}  Let $(x_a,x_b,x_c,x_d)$ be a point in $\Delta(2,4)$.
Define $ \mathbf{d}(x_a,x_b,x_c,x_d)$ to be the distance from
$(x_a,x_b,x_c,x_d)$ to the boundary of $\Delta(2,4)$.  Then $L_{\x}
\cdot F_{a,b,c,d}$ is a multiple of $\mathbf{d}(x_a,x_b,x_c,x_d).$

\subsection{Results on the GIT cone} 
Suppose $\x \in \Delta(2,n)$ with $x_{i} \neq 0$ for any $i$. The GIT
stability criterion for points on a line is well-known:
$(P_{1},\ldots,P_{n}) \in (\Pro^{1})^{n}$ is $SL(2)$-stable (resp.
semistable) with respect to the linearization $\x$ if and only if for
every subset $I \subseteq \{1,\ldots,n\}$ for which the points $\{P_i
\mid i \in I\}$ coincide, $\sum_{i \in I} x_{i} < 1$ (resp. $\leq 1$).
Thus, for a generic linearization, stability and semistability
coincide; strictly semistable points only occur for those $\x$ which
lie on a hyperplane of the form $\sum_{i \in I} x_{i} = 1$ for some
subset $I \subseteq \{1,\ldots,n \}$.  This is the familiar
wall-and-chamber decomposition of the space of linearizations, due to
Dolgachev--Hu and Thaddeus.  We call the subdivision of $\Delta(2,n)$
given by 
these hyperplanes  \textit{the GIT complex}.

The GIT cone is generated by the 0-cells of the GIT complex
because $L_{\x} \cdot F_{a,b,c,d}$ is linear on each
chamber.

For $n=5,6,7,8,9$, the 0-cells of the GIT-complex were computed using a C++ program written for us by Boris Alexeev.  This data is available on the second author's website:
\begin{center} \begin{verbatim}http://www.math.uga.edu/~davids/research.html\end{verbatim}\end{center}
For $n=10$, this calculation became intractable.

For $n=5$, we can easily check that the GIT cone is the nef cone.  The 0-cells for $n=5$ are given by $(\frac{2}{3},\frac{1}{3},\frac{1}{3},\frac{1}{3})$, $(\frac{1}{2},\frac{1}{2},\frac{1}{2},\frac{1}{2},0)$, $(1,1,0,0,0)$, and permutations of these.  The line bundles of type $(1,1,0,0,0)$ are trivial. The line bundles of type $(\frac{2}{3},\frac{1}{3},\frac{1}{3},\frac{1}{3})$ each contract four disjoint curves, giving $\Pro^{2}$.  The line bundles of type $(\frac{1}{2},\frac{1}{2},\frac{1}{2},\frac{1}{2},0)$ give morphisms to $\Pro^{1}$.  

But for $n \geq 6$ we see that the GIT cone is strictly contained in
the nef cone.  Indeed, when $n=6$, Faber computes that the nef cone
has $3190$ extremal rays (\cite{F}).  We computed the 0-cells of the
GIT complex by a search implemented in \mtwo.  The interior 0-cells
are $(\frac{1}{3},\frac{1}{3},\frac{1}{3},\frac{1}{3},\frac{1}{3},\frac{1}{3})$, $(\frac{1}{4},\frac{1}{4},\frac{1}{4},\frac{1}{4},\frac{1}{4},\frac{3}{4})$,
$(\frac{1}{4},\frac{1}{4},\frac{1}{4},\frac{1}{4},\frac{1}{2},\frac{1}{2})$, $(\frac{3}{5}, \frac{2}{5},\frac{2}{5},\frac{1}{5},\frac{1}{5},\frac{1}{5})$ and their
symmetric images, and the boundary 0-cells are $(1,1,0,0,0,0)$,
$(\frac{1}{2},\frac{1}{2},\frac{1}{2},\frac{1}{2},0,0)$, and 
\linebreak
$(\frac{2}{3},\frac{1}{3},\frac{1}{3},\frac{1}{3},\frac{1}{3},0)$ and their symmetric images.  There are only 142 0-cells (16 of these give the trivial line bundle), and they do not generate the nef cone.  

Interestingly, we found that for $n=5,6,7,8,9$, if $\x$ is a 0-cell, then $L_{\x}$ is extremal in $\Nef(\M_{0,n})$.  This leads to a natural question:

\begin{question}
If  $\x$ is a 0-cell of the GIT complex, is $L_{\x}$ extremal in $\Nef(\M_{0,n})$?
\end{question}

\subsection{Notation for symmetrized line bundles} 
It is natural to work with symmetrizations of the $L_{\x}$.    We use the following notation for these: 

\begin{definition} \label{Van def}
We denote  the symmetrization of any $L_{\x}$ by $S_{\x}$:
\[   S_{\x} := \bigotimes_{\sigma \in S_{n}} L_{\sigma \x}.
\]

For some special sets of weights, we have additional notation.  Let $a$ be a rational number such that $\frac{1}{n-1} \leq a \leq \frac{2}{n-1}$, and write $\bar{a} = 2-(n-1)a$.  This condition on $a$ ensures that $0\leq a \leq \bar{a} \leq 1$.  We write $L(a,i)$ for $L_{\x}$ where $x_{j}= a$ if $j \neq i$ and $x_{i} =\bar{a}$, and call the $i^{th}$ entry the \emph{odd entry}.  We write 
\[  \V(a,n) := \bigotimes_{i=1}^{n} L(a,i)
\] 
Thus, $\V(a,n)$ is a reduced symmetrization: 
\[ \V(a,n) =\frac{1}{(n-1)!} S_{\x}.
\] 
\end{definition}
Note that when $a=2/n$, we have $a=\bar{a}$; our convention in this case is that $\V(a,n) = L_{(a,...,a)}^{\otimes n}$ (that is, we do not further reduce the symmetrization).

\subsection{Results on the symmetrized GIT cone}

When working with the symmetrized GIT cone, we have frequently been able to obtain good results using only a few of the $V(a,n)$'s.  Consider the simplicial cone 
\[\left\langle V(a,n) \mid 
  a=\dfrac1{t}, \ t\in \mathbb Z,\ \left\lfloor\dfrac{n}{2}\right\rfloor \le t\le n-2
  \right\rangle 
\] 
\textit{A priori} this is a subcone of the $\SGC$, but we have checked for $n=5,6,7,8,9$ that this subcone is actually equal to the $\SGC$.  This raises a natural question:
\begin{question} \label{SGCconj} Is $\SGC = \left\langle V(a,n) \mid 
  a=\dfrac1{t}, \ t\in \mathbb Z,\ \left\lfloor\dfrac{n}{2}\right\rfloor \le t\le n-2
  \right\rangle $ for all $n \geq 5$?
\end{question}

For $n=6$, we computed $S_{\x}$ for the 0-cells above, and found that the $\SGC$ is generated by $V(\frac{1}{3},6)$ and $V(\frac{1}{4},6)$.   
For symmetric divisors, the shape of the partition $\{a,b,c,d \}$ determines the intersection $D \cdot F_{a,b,c,d}$, so we need only record the intersections with $F_{3,1,1,1}$ and $F_{2,2,1,1}$.  These are as follows:
\begin{displaymath}
\begin{array}{lcc} & \V(\frac{1}{3},6) & \V(\frac{1}{4},6)  \\ 
F_{2,2,1,1} & 4 & 1  \\ 
F_{3,1,1,1}& 0 & 3/2 
\end{array}
\end{displaymath}
Then if $(S_{\x} \cdot F_{2,2,1,1}, S_{\x} \cdot F_{3,1,1,1}) = (a,b)$, we can write $S_{\x} \equiv c_{1} V(\frac{1}{3},6) + c_{2} V(\frac{1}{4},6)$, where 
\begin{displaymath}
\left(\begin{array}{cc}
c_{1}\\
c_{2}
\end{array} \right)
 = 
\left(\begin{array}{cc}
4 & 1 \\
0 & 3/2
\end{array}\right)^{-1}
\left(\begin{array}{c}
a\\
b
\end{array}\right)
= 
\left(\begin{array}{c}
\frac{1}{4}a - \frac{1}{6}b\\
\frac{2}{3}b
\end{array}\right).
\end{displaymath}
$\frac{2}{3}b$ is automatically nonnegative, since $S_{\x}$ is nef.  So the conjecture will be true if $\frac{1}{4}a - \frac{1}{6}b \geq 0$ for any $\x$.  But 
\[ F(\x) = \frac{1}{4}a - \frac{1}{6}b = \frac{1}{4}(S_{\x} \cdot F_{2,2,1,1}) - \frac{1}{6}(S_{\x} \cdot F_{3,1,1,1})
\]
is a piecewise linear function which breaks only along the hyperplanes $\sum_{i \in I} x_{i} = 1$.  Hence $F$ attains its minimum on a 0-cell of the GIT complex.  We check all the 0-cells for $n=6$ and see that $F \geq 0$ for all of them.  

We have answered Question \ref{SGCconj} affirmatively for $n=5,6,7,8,9$, where we have a full list of 0-cells.

We give an example showing that $\SGC$ is strictly contained in the symmetric nef cone for $n=6$.  The F-conjecture is known for $n=6$ (\cite{F}, \cite{KM}, \cite{GKM}), so any line bundle which intersects the $F$-curves nonnegatively is nef.  Suppose $D:= r_{2} D_{2} + r_{3} D_{3}$ is a symmetric divisor on $\M_{0,6}$.  Then by \cite{KM} Corollary 4.4, $(D \cdot F_{2,2,1,1}, D \cdot F_{3,1,1,1}) = (2r_{3}-r_{2},3r_{2}-r_{3})$, and $D$ is nef provided $2r_{3}-r_{2} \geq 0 $ and $3r_{2}-r_{3} \geq 0$, or $2r_{3} \geq r_{2} \geq (\frac{1}{3})r_{3}$.  But there are values of $(r_{2},r_{3})$ for which $2r_{3} \geq r_{2} \geq (\frac{1}{3})r_{3}$ but $\frac{1}{4}a - \frac{1}{6}b  =  8r_{3} - 9r_{2} < 0$. (For instance, $(r_{2},r_{3}) = (1,1)$.)  So for $n =6$ the $\SGC$ is strictly contained in the symmetric nef cone.

\subsection{Example: $\psi_i$ is a GIT line bundle}
\label{sec:psi-is-GIT}
Here we show that the divisor $\psi_{i}$ on $\overline{M}_{0,n}$ is a
multiple of the GIT divisor $L_{\x}$ for the weights 
\begin{equation}
x_j = \left\{
\begin{array}{ll}
\rule{0pt}{22pt} \displaystyle \frac{1}{n-2} & \mbox{ if $j \neq i$} \\
\rule{0pt}{22pt} \displaystyle \frac{n-3}{n-2} & \mbox{ if $j=i$}
\end{array} \right.
\end{equation}

We show that the intersection numbers agree (up to a constant multiple).

\begin{proposition}  \label{prop:intersection-psi-F}
The intersection numbers of $\psi_i$ with $F$-curves are as follows:
\begin{equation}
F_{S_{1},S_{2},S_{3},S_{4}} \cdot \psi_{i} = \left\{
\begin{array}{ll}
1 & \mbox{ if $S_{k}= \{i\}$ for some $k$} \\
0 & \mbox{ otherwise}
\end{array} \right.
\end{equation}
\end{proposition}
\begin{proof} This follows from the formula for $\psi_i$ given in \cite{AC} Prop. 3.6.
\end{proof}

Indeed, by Lemma \ref{phi lemma}, if $S_k$ contains $i^{th}$ point and any other point, then we will have $x_k \geq 1$, and hence $L_{\x} \cdot F_{S_{1},S_{2},S_{3},S_{4}} =0$.  On the other hand, suppose $S_{1} = \{i\}$, and $\#S_{2} \geq \#S_{3} \geq \#S_{4}$.   Then $\#S_{2} \leq n-3$, so $x_{b} \leq \frac{n-3}{n-2}$.  Also $\#S_{3} , \#S_{4} \geq 1$, so $x_c, x_d \geq \frac{1}{n-2}$.  Thus $\max = \frac{n-3}{n-2}$ and $\max + \min \geq 1$, so $L_{\x}\cdot F_{S_{1},S_{2},S_{3},S_{4}} = \frac{2}{n-2}$.  

In the symmetrized case, we have $\Psi = \sum_{i=1}^{n} \psi_i $ is a multiple of $V(\frac{1}{n-2},n)$.

\section{Simpson's divisors are effective combinations of the
$V(a,n)$} \label{rewriteAalpha}
%\showlabelsbreak

\subsection{The divisors $A_{\alpha}$} \label{Bibasis} First, we recall the key
divisors $A_{\alpha}$ from Simpson's work.  Write $B_{j} = \sum_{|S| =
j} B_{S}$.  Then one basis for the $S_{n}$-equivariant divisors on
$\M_{0,n}$ is $\{ B_{2}, \ldots, B_{\lfloor \frac{n}{2} \rfloor} \}$.
Simpson establishes a formula for the $A_{\alpha}$ in this basis,
which we will use as the definition.

\begin{definition}[cf. \cite{SimpA} Def. 2.3.2 and Lemma 2.3.1]
\label{Aalpha def} Write
\begin{equation} k(\alpha) := \left\lfloor \frac{2}{\alpha}
\right\rfloor -1.
\end{equation} Then
\begin{equation} \label{Aalpha eqn} A_{\alpha} = \sum_{j=2}^{k} \left(
\binom{j}{2} \alpha - \frac{j(j-1)}{n-1} \right) B_{j} +
\sum_{j=k+1}^{\lfloor \frac{n}{2} \rfloor } \left( \frac{(j-2)(n-1)
-j(j-1)}{n-1} + \alpha \right) B_{j}.
\end{equation}
\end{definition}

Simpson observes that if $\alpha \in [\frac{2}{k+2}, \frac{2}{k+1}]$,
then $A_{\alpha}$ is a convex combination of $A_{2/(k+2)}$ and
$A_{2/(k+1)}$, and hence:

\begin{claim} \label{critical alphas} It is sufficient to show that
$A_{\alpha}$ is nef for $\alpha = \frac{2}{k+1}$ where $k = 2,\ldots,
\lfloor \frac{n}{2} \rfloor -1$.  We call these the \emph{critical}
$\alpha$ and $A_{\alpha}$.
\end{claim}

\textit{Remark.}  It is easy to check that $A_{2/3} =
\frac{1}{3}A_{1}+ \frac{2}{3}A_{1/2}$.  However, this will not affect
any of the analysis below, and it does not appear that any of the
other critical $A_{\alpha}$ is a convex or effective combination of
the others.

\subsection{When $n$ is odd} Suppose $n$ is odd.  It is convenient to use the following
notation.
\begin{definition} Write
\begin{eqnarray} f & := & \left\lfloor \frac{n}{2} \right\rfloor =
\frac{n-1}{2}\\ \ell & := & f+1-k.
\end{eqnarray}
\end{definition} We chose the letter $f$ for \underline{f}loor.

\begin{proposition} \label{exp comb} Suppose $n$ is odd, and that
$\alpha = \frac{2}{k+1}$ for some $k$ in $\{1,\ldots, \lfloor (n-1)/2
\rfloor \}$.  Hence $\alpha = \frac{2}{f-\ell+2}$, and $\ell \in \{
1,\ldots, f -1\}$.  Then
\begin{equation} A_{\alpha} \equiv c_{1} \V \left(\frac{1}{f+1}, n
\right) + \cdots + c_{f-1} \V \left( \frac{1}{2f-1},n \right)
\end{equation} where
\begin{equation} \label{first ci} c_{i} = \alpha
\frac{(f+i)(f-\ell+1)}{(f-i)(f-i+1)(f-i+2)}
\end{equation} if $1 \leq i \leq \ell-1$,
\begin{equation} \label{c ell} c_{\ell} = \frac{1}{4} \alpha
\frac{(f+\ell)(f-\ell)}{f-\ell+2},
\end{equation} and $c_{i} = 0$ for $\ell+1 \leq i \leq f-1$.  This
covers all critical $\alpha$ except $\alpha =1$.

For $\alpha = 1$, we have
\begin{equation} \label{alpha=1 ci} c_{i} =
\frac{f+i}{(f-i)(f-i+1)(f-i+2)}
\end{equation} for $1 \leq i \leq f -1$.

The coefficients $c_i$ are nonnegative in all cases. 
\end{proposition}

\begin{proof} We choose a set of $F$-curves $C_{1}, \ldots, C_{f-1}$
(see Definition \ref{Cidef} below) which give a full rank intersection
matrix ($C_{i} \cdot \V(\frac{1}{f+j},n))$ for symmetric divisors on
$\M_{0,n}$.  Then the $c_{i}$ are the solutions of the system
\begin{equation} 
%\showlabelsbreakmath
\label{intersectionsystem} 
\left(C_{i} \cdot \V
\left( \frac{1}{f+j},n \right) \right) \vec{c} = \left( C_{i} \cdot
A_{\alpha} \right).
\end{equation}

In Lemmas \ref{CidotVanlemma} and \ref{CidotAalphalemma} below, we
show that the system (\ref{intersectionsystem}) takes on the following
form:
\begin{equation} \label{uppertriangularsystem} \left(
\begin{array}{ccccccc} \frac{2(f+1)}{f+1} & \frac{4}{f+2} &
\frac{4}{f+3} & \cdots & \frac{4}{f+j} & \cdots & \frac{4}{2f-1}\\ 0 &
\frac{2f}{f+2} & \frac{4}{f+3} & \vdots & \vdots & \vdots &
\frac{4}{2f-1}\\ 0 & 0 & \frac{2(f-1)}{f+3} & \vdots & \vdots & \vdots
& \frac{4}{2f-1} \\ \vdots & \vdots & 0 & \ddots & \vdots & \vdots &
\vdots \\ \vdots & \vdots & \vdots & 0 & \frac{2(f-j+2)}{f+j} & \vdots
& \vdots \\ \vdots & \vdots & \vdots & \vdots & \ddots & \ddots &
\frac{4}{2f-1} \\ 0 & 0 & 0 & 0 & 0 & 0 & \frac{6}{2f-1}
\end{array} \right ) \left(
\begin{array}{c} c_{1}\\ c_{2}\\ \vdots\\ \vdots\\ \vdots\\ \vdots\\
c_{f-1}
\end{array} \right) = \left(
\begin{array}{c} \alpha\\ \vdots\\ \alpha\\ 1- \alpha\\ 0\\ \vdots\\ 0
\end{array} \right)
\begin{array}{c} \mbox{} \\ \mbox{} \\ \mbox{} \mbox{\tiny{$
\leftarrow (\ell-1)^{th}$ coordinate}}\\ \mbox{}
\mbox{\tiny{$\leftarrow \ell^{th}$ coordinate}}\\ \mbox{} \\ \mbox{}
\\ \mbox{}
\end{array}
\end{equation}

We will refer to the upper triangular matrix on the left hand side of
(\ref{uppertriangularsystem}) as $U$ and write $u_{i,j}$ for its
entries.  The vector shown on the right hand side of
(\ref{uppertriangularsystem}) is for $\alpha < \frac{2}{3}$.  For $\alpha =
\frac{2}{3}$, the vector $(C_{i} \cdot A_{\alpha})$ on the right hand side has
no trailing zeroes, and for $\alpha=1$, the right hand side vector is
$(1,\ldots,1)^{T}$.

Thus, to prove Proposition \ref{exp comb}, it remains only to show
that the $c_{i}$ given in the statement of the proposition are indeed
the solutions of the system (\ref{uppertriangularsystem}).  This is
done in Lemmas \ref{ratlfcnlemma} and \ref{ciaresolns}.
\end{proof}

\begin{definition} \label{Cidef} The set of $F$-curves we use to
define the $(f-1) \times (f-1)$ system above is as follows:
\begin{eqnarray} C_{1} & := & F_{f,f-1,1,1} \nonumber \\ 
C_{2} & := & F_{f+1, f-2, 1,1} \nonumber \\ 
& \vdots & \nonumber \\ 
C_{i} & := &
F_{f+i-1,f-i,1,1} \nonumber \\ & \vdots & \nonumber \\ 
C_{f-1} & := &
F_{f+1+(f-1),f-(f-1),1,1} = F_{n-3,1,1,1} \label{F curve basis}
\end{eqnarray} where for each $i$ we may choose any partition of $\{
1,\ldots,n \}$ which has the indicated shape.
\end{definition}

The next lemma establishes that the intersection matrix $C_{i} \cdot
\V \! \left( \frac{1}{f+j},n \right)$ is the matrix $U$ shown in
(\ref{uppertriangularsystem}).

\begin{lemma}\label{CidotVanlemma} For any $i, j \in \{1,\ldots,f-1
\}$,
\begin{displaymath} u_{i,j} := C_{i} \cdot \V\left( \frac{1}{f+j},n \right) =
\left\{
\begin{array}{ll} \rule{0pt}{14pt}\displaystyle \frac{4}{f+j} & \mbox{if $i< j$}\\
\displaystyle \rule{0pt}{14pt} \frac{2(f-j+2)}{f+j} & \mbox{if $i=j$}\\ 
\rule{0pt}{14pt} \displaystyle 0 &\mbox{if $i > j$}
\end{array} \right.
\end{displaymath}
\end{lemma}
\begin{proof} Note that the odd entry of $\V(\frac{1}{f+j},n)$ is
$2-\frac{n-1}{f+j} = \frac{2j}{f+j}$.  We compute the intersection of $C_{i}$
with each $L_{\x}$ constituting $\V(\frac{1}{f+j},n)$ and sum these to
obtain $C_{i} \cdot \V(\frac{1}{f+j},n)$.

\textit{Case 1: $i<j$.}  First suppose the odd entry of $L_{\x}$ lands
in the long tail of $C_{i}$.  Then, using the notation of Lemma
\ref{phi lemma},
\[ x_{a} = (f+i) \frac{1}{f+j} + \frac{2j}{f+j} > 1,
\] so $C_{i} \cdot L_{\x}=0$.

Next suppose the odd entry of $L_{\x} $ is on the second tail of
$C_{i}$.  Then
\[ x_{b} = (f-i-1)\frac{1}{f+j} + \frac{2j}{f+j} \geq 1,
\] so $C_{i} \cdot L_{\x}=0$.

Thus we see that $L_{\x}$ contributes to $C_{i} \cdot \V(\frac{1}{f+j},n)$
only if the odd entry is on the spine of $C_{i}$.  The maximum comes
from either the long tail, which has weight $x_{a} = \frac{f+i-1}{f+j}$,
or the odd entry, which has weight $x_{c} =\frac{2j}{f+j}$.  It seems that
both are possible.  The minimum is $x_{d} = \frac{1}{f+j}$.  However, since
both $x_{a}<1$ and $x_{c} < 1$, and both $x_{a} + \min \leq 1$ and
$x_{b} + \min \leq 1$, we find that $C_{i} \cdot L_{\x} = 2 \min=
\frac{2}{f+j}$.

Thus, $C_{i} \cdot \V(\frac{1}{f+j},n) = 2 \cdot \frac{2}{f+j} = \frac{4}{f+j}$.

\textit{Case 2: $i=j$.}  If the odd entry is on the long tail,
\[ x_{a}=(f-i-2) \frac{1}{f+j} + \frac{2j}{f+j} = \frac{f+3j-2}{f+j} >
1,
\] so $C_{i} \cdot L_{\x}=0$.

If the odd entry is on the second tail,
\[ x_{b}= (f-j-1) \frac{1}{f+j} +\frac{2j}{f+j}=\frac{f+j-1}{f+j} < 1.
\] Meanwhile, on the long tail,
\[ x_{a} = \frac{f+j-1}{f+j},
\] so in the notation of Lemma \ref{phi lemma} $\max = x_{a} = x_{b}$,
and $\max+\min =1$, so $C_{i} \cdot L_{\x} = 2\min = \frac{2}{f+j}$.

Finally, if the odd entry is on the spine, then the long tail gives
the maximum, and $\max+\min = 1$, so $C_{i} \cdot L_{\x} = 2\min =
\frac{2}{f+j}$.

Thus,
\[ C_{i} \cdot \V(\frac{1}{f+j},n) = (f-i) \frac{2}{f+j} + 2\frac{2}{f+j}
=\frac{2(f-j+2)}{f+j}
\].

\textit{Case 3: $i > j$.}  If the odd entry is not on the long tail,
then
\[ x_{a} = (f+i-1) \frac{1}{f+j} \geq 1
\] and $C_{i} \cdot L_{\x} =0$.

If the odd entry is on the long tail, then $x_{a}$ is even bigger, and
$C_{i} \cdot L_{\x} =0$.  Thus, $C_{i} \cdot \V(\frac{1}{f+j},n) =0$.
\end{proof}

In the next lemma we compute the right hand side of
(\ref{uppertriangularsystem}).

\begin{lemma} \label{CidotAalphalemma} Let $\alpha =
\frac{2}{f-\ell+2}$ for $\ell \in \{ 1,\ldots, f -1\}$.  Then the
intersection numbers $C_{i} \cdot A_{\alpha}$ are as follows:
\begin{equation} C_{i} \cdot A_{\alpha} = \left\{
\begin{array}{c} \alpha \mbox{ if $i<\ell$,}\\ 1- \alpha \mbox{ if
$i=\ell$,}\\ 0 \mbox{ if $i > \ell$}.
\end{array} \right.
\end{equation} For $\alpha = 1$, we have $C_{i} \cdot A_{1} = 1$ for
all $i$.
\end{lemma}
\begin{proof} It is lengthy but straightforward to compute this using
Definition \ref{Aalpha def} and \cite{KM} Corollary 4.4.
\end{proof}

We have now identified the systems (\ref{intersectionsystem}) and
(\ref{uppertriangularsystem}).

Next we prove that the $c_{i}$ given in the statement of Proposition
\ref{exp comb} are indeed the solutions of this system.  We will use
the following identity:

\begin{lemma}\label{ratlfcnlemma} The following identity holds for
rational functions of a single variable $y$:
\begin{equation} \label{ratlfcnid} \sum_{p=1}^{m}
\frac{1}{(y-p)(y-p+1)(y-p+2)} = \frac{2ym-m^2+m}{2y(y+1)(y-m)(y-m+1)}.
\end{equation}
\end{lemma}
\begin{proof} By induction on $m$.  The induction step is easily
verified by hand or with a computer algebra system.
\end{proof}

\begin{lemma} \label{ciaresolns} The $c_{i}$ defined in (\ref{first
ci}) and (\ref{c ell}) for $\alpha <1$, or in (\ref{alpha=1 ci}) for
$\alpha=1$, are the solutions of the system
(\ref{uppertriangularsystem}).
\end{lemma}
\begin{proof} We give the proof for $\alpha < 1$.  The proof for
$\alpha=1$ is similar.

It is easy to verify that
\begin{eqnarray} u_{\ell,\ell} c_{\ell} & = & \frac{f-\ell}{f-\ell+2}
\label{uellell}\\ u_{\ell-1,\ell-1} c_{\ell-1} + u_{\ell-1,\ell}
c_{\ell} & = & \alpha. \nonumber
\end{eqnarray} Next we consider any $i < \ell-1$.  We wish to show
\begin{eqnarray} u_{i,i} c_{i} + u_{i,i+1}c_{i+1} + \cdots +
u_{i,\ell-1} c_{\ell-1} + u_{i,\ell} c_{\ell} & = & \alpha \nonumber
\\ \Leftrightarrow u_{i,i} c_{i} + \sum_{p=1}^{\ell-1-i} u_{i,i+p}
c_{i+p} + u_{i,\ell} c_{\ell} & = & \alpha \nonumber\\ \Leftrightarrow
u_{i,i} c_{i} + \sum_{p=1}^{\ell-1-i} u_{i,i+p} c_{i+p} & = & \alpha -
u_{\ell,\ell} c_{\ell} \nonumber\\ \Leftrightarrow u_{i,i} c_{i} +
\sum_{p=1}^{\ell-1-i} u_{i,i+p} c_{i+p}& = & \alpha^{2},
\label{toshow}
\end{eqnarray} where we have used our computation of
$u_{\ell,\ell}c_{\ell}$ in (\ref{uellell}) above and the definition of
$\alpha$ to obtain the last line (\ref{toshow}).

Now we substitute $u_{i,j}$ and $c_{j}$ into (\ref{toshow}):
\begin{eqnarray} \lefteqn{\frac{2(f\! - \!i\! + \!2)}{f\! + \!i}
\alpha \frac{(f\! + \!i)(f\! - \!\ell\! + \!1)}{(f\! - \!i)(f\! -
\!i\! + \!1)(f\! - \!i\! + \!2)} \! + \! \sum_{p=1}^{\ell\! - \!1\! -
\!i} \frac{4}{f\! + \!i\! + \!p} \alpha \frac{(f\! + \!i\! + \!p)(f\!
- \!\ell\! + \!1)}{(f\! - \!i\! - \!p)(f\! - \!i\! - \!p\! + \!1)(f\!
- \!i\! - \!p\! + \!2)} = \alpha^{2}} \nonumber \\ \Leftrightarrow
&&\frac{2(f\! - \!\ell\! + \!1)}{(f\! - \!i)(f\! - \!i\! + \!1)} \! +
\! \sum_{p=1}^{\ell\! - \!1\! - \!i} \frac{4(f\! - \!\ell\! +
\!1)}{(f\! - \!i\! - \!p)(f\! - \!i\! - \!p\! + \!1)(f\! - \!i\! -
\!p\! + \!2)} = \frac{2}{f\! - \!\ell\! + \!2} \nonumber \\
\Leftrightarrow && \frac{1}{(f\! - \!i)(f\! - \!i\! + \!1)} \! + \!
\sum_{p=1}^{\ell\! - \!1\! - \!i} \frac{2}{(f\! - \!i\! - \!p)(f\! -
\!i\! - \!p\! + \!1)(f\! - \!i\! - \!p\! + \!2)} = \frac{1}{(f\! -
\!\ell\! + \!1)(f\! - \!\ell\! + \!2)} \nonumber \\ \Leftrightarrow &&
\sum_{p=1}^{\ell\! - \!1\! - \!i} \frac{2}{(f\! - \!i\! - \!p)(f\! -
\!i\! - \!p\! + \!1)(f\! - \!i\! - \!p\! + \!2)} = \frac{1}{2} \left(
\frac{1}{(f\! - \!\ell\! + \!1)(f\! - \!\ell\! + \!2)} \! - \!
\frac{1}{(f\! - \!i)(f\! - \!i\! + \!1)} \right). \label{fell}
\end{eqnarray}

We substitute $y=f-i$ and $m=\ell-1-i$.  Then (\ref{fell}) becomes
\[ \sum_{p=1}^{m} \frac{1}{(y-p)(y-p+1)(y-p+2)} = \frac{1}{2} \left(
\frac{1}{(y-m)(y-m+1)} - \frac{1}{y(y+1)} \right),
\] which reduces to the identity (\ref{ratlfcnid}). \end{proof}

We have now proved all the lemmas used in the proof of Proposition
\ref{exp comb}.

\subsubsection{Example: $n=9$.}  Here $f = 4$.  The four critical
$\alpha$ are $2/5$ (which corresponds to $(\Pro^{1})^{9} \dblq
SL(2)$), $\frac{1}{2}$, $\frac{2}{3}$, and $1$ (which corresponds to $\M_{0,9}$).  We
use the three $F$-curves $F_{4,3,1,1}$, $F_{5,2,1,1}$, and
$F_{6,1,1,1}$ and the line bundles $\V(\frac{1}{5},9)$, $\V(\frac{1}{6},9)$,
$\V(\frac{1}{7},9)$ to form the intersection matrix:
\begin{displaymath}
\begin{array}{lccc} & \V(\frac{1}{5},9) & \V(\frac{1}{6},9) & \V(\frac{1}{7},9) \\ F_{4,3,1,1}
& 2 & \frac{2}{3} & 4/7 \\ F_{5,2,1,1} & 0 & \frac{4}{3} & \frac{4}{7} \\ F_{6,1,1,1} & 0 & 0 &
\frac{6}{7}
\end{array}
\end{displaymath} and the four vectors $C_{i} \cdot A_{\alpha}$ are
\begin{displaymath}
\begin{array}{rcccc}\alpha: & \frac{2}{5}& \frac{1}{2} & \frac{2}{3} & 1 \\ F_{4,3,1,1} \cdot
A_{\alpha}: & \frac{3}{5} & \frac{1}{2} & \frac{2}{3} & 1 \\ F_{5,2,1,1} \cdot A_{\alpha}: & 0
& \frac{1}{2} & \frac{2}{3} & 1 \\ F_{6,1,1,1} \cdot A_{\alpha}: & 0 & 0 & \frac{1}{3} & 1
\end{array}.
\end{displaymath} This leads to the following equivalences:
\begin{eqnarray} A_{2/5} & \equiv & \frac{3}{10} \V(\frac{1}{5},9) \nonumber
\\ A_{1/2} & \equiv & \frac{1}{8} \V(\frac{1}{5},9) + \frac{3}{8} \V(\frac{1}{6},9)
\nonumber \\ A_{2/3} & \equiv & \frac{1}{9} \V(\frac{1}{5},9) + \frac{1}{3}
\V(\frac{1}{6},9) + \frac{7}{18} \V(\frac{1}{7},9) \nonumber \\ A_{1} & \equiv &
\frac{1}{12} \V(\frac{1}{5},9) + \frac{1}{4} \V(\frac{1}{6},9) + \frac{7}{6}
\V(\frac{1}{7},9). \nonumber
\end{eqnarray}

\subsection{When $n$ is even} We obtain similar results by similar
methods.

Once again write $f= \lfloor n/2 \rfloor = n/2$ since $n$ is even.  We
use the curves $C_{i} = F_{f-2+i,f-i,1,1}$ for $i=1,\ldots,f-1$ as our
basis for the intersection matrix.  The resulting system is
\begin{equation} \label{evensystem} \left( \begin{array}{ccccccc} 4 &
\frac{4}{f+1} & \frac{4}{f+2} & \cdots & \frac{4}{f+j-1} & \cdots &
\frac{4}{2f-2}\\ 0 & \frac{2f}{f+1} & \frac{4}{f+2} & \vdots & \vdots
& \vdots & \vdots\\ 0 & 0 & \frac{2(f-1)}{f+2} & \vdots & \vdots &
\vdots & \vdots \\ \vdots & \vdots & 0 & \ddots & \vdots & \vdots &
\vdots \\ \vdots & \vdots & \vdots & 0 & \frac{2(f-j+2)}{f+j-1} &
\vdots & \vdots \\ \vdots & \vdots & \vdots & \vdots & \ddots & \ddots
& \frac{4}{2f-2} \\ 0 & 0 & 0 & 0 & 0 & 0 & \frac{6}{2f-2}
\end{array} \right ) \left(
\begin{array}{c} c_{1}\\ c_{2}\\ \vdots\\ \vdots\\ \vdots\\ \vdots\\
c_{f-1}
\end{array} \right) = \left(
\begin{array}{c} \alpha\\ \vdots\\ \alpha\\ 1-\alpha\\ 0\\ \vdots\\ 0
\end{array} \right),
\end{equation} and it has the solution given below.  Note that more
formulas are required when $n$ is even than when $n$ is odd.  We may
attribute this to the fact that $C_{1}$ and $\V(\frac{2}{n},n)$ are more
symmetric than any other pair under consideration, and as a result
$u_{1,1}$ does not fit the pattern observed in the other diagonal
entries of the intersection matrix when $n$ is even.

\begin{proposition} \label{n even exp comb} Suppose $n \geq 6$ is even.

If $\alpha = \frac{2}{f+1}$ (which corresponds to $(\Pro^{1})^{n}
\dblq SL(2)$), then
\[ A_{\alpha} \equiv (f-2)/(2f+2) \V (\frac{1}{f},n).
\]

If $\alpha = \frac{2}{f}$, then
\[ A_{\alpha} \equiv \frac{1}{f^2} \V \left( \frac{1}{f},n \right) +
\frac{(f+1)(f-2)}{2f^2} \V \left( \frac{1}{f+1},n \right).
\]

If $\alpha = \frac{2}{f-\ell+2}$ for some $\ell \in \{ 3, \ldots, f-1
\}$, then
\[ A_{\alpha} \equiv c_{1} \V \left( \frac{1}{f},n \right) + \cdots
c_{f-1} \V \left( \frac{1}{2f-2},n \right),
\] where
\begin{eqnarray} c_{1} & = & \frac{1}{2} \alpha
\frac{f-\ell+1}{f(f-1)}, \nonumber \\ c_{i} & = & \alpha
\frac{(f-\ell+1)(f-1+i)}{(f-i)(f-i+1)(f-i+2)} \mbox{ if $2 \leq i \leq
\ell-1$,} \nonumber \\ c_{\ell} & = & \frac{1}{4} \alpha
\frac{(f-\ell)(f+\ell-1)}{f-\ell+2}, \nonumber \\ c_{i} & = & 0 \mbox{
if $i > \ell$}.
\end{eqnarray}

If $\alpha =1$ and $n=6$ (which corresponds to $\M_{0,6}$), then
\[ A_{\alpha} \equiv \frac{1}{12} \V \left( \frac{1}{3}, 6 \right) +
\frac{2}{3} \V\left( \frac{1}{4}, 6 \right).
\]

If $\alpha =1$ and $n \geq 8$ (which corresponds to $\M_{0,n}$), then
\[ A_{\alpha} \equiv c_{1} \V \left( \frac{1}{f},n \right) + \cdots
c_{f-1} \V \left( \frac{1}{2f-2},n \right)
\] where
\begin{eqnarray} c_{1} & = & \frac{1}{2f(f-1)}, \nonumber \\ c_{i} & =
& \frac{(f-1+i)}{(f-i)(f-i+1)(f-i+2)} \mbox{ if $2 \leq i \leq f-2$,}
\nonumber \\ c_{f-1} & = & \frac{f-1}{3}.
\end{eqnarray}

The coefficients $c_i$ are nonnegative in all cases. 
\end{proposition}

\section{Nefness and ampleness of $A_{\alpha}$} \label{amplesection}

Since each critical $A_{\alpha}$
is an effective combination of the $L_{\x}$, and every $A_{\alpha}$ is
a convex combination of the critical $A_{\alpha}$, we immediately obtain the
following corollary:

\begin{corollary} \label{Aalphaisnef} $A_{\alpha}$ is nef on $\M_{0,n}$ for all
$\frac{4}{n+1} \leq \alpha \leq 1$.
\end{corollary}

By inspecting the formulas of the previous section more closely, we
will obtain a more precise result on the positivity of the
$A_{\alpha}$ in Corollary \ref{ample on M_0,beta}.

\begin{lemma}  Write $f =
\lfloor \frac{n}{2} \rfloor$.  Let $k \in \{ 2, \ldots, f\}$ and write $\beta = (\frac{1}{k}, \ldots, \frac{1}{k})$ and 
$\ell = f+1-k$ as in the previous section.
\begin{enumerate}
\item If $n$ is odd, and $F_{a,b,c,d}$ is an F-curve with $\#a \geq
  \#b \geq \#c \geq \#d$ and $\#a \geq n-k$, then 
$V(\frac{1}{f+j},n) \cdot F_{a,b,c,d} = 0$ for $j=1,\ldots, \ell$.
\item If $n$ is even, and $F_{a,b,c,d}$ is an F-curve with $\#a \geq
  \#b \geq \#c \geq \#d$ and $\#a \geq n-k$, then 
$V(\frac{1}{f+j-1},n) \cdot F_{a,b,c,d} = 0$ for $j=1,\ldots, \ell$.
\item  $\dim \Pic(\M_{0,\beta})^{S_n}  = \ell$.
\end{enumerate}
\end{lemma}
\begin{proof}
Recall from Definition \ref{Van def} that $V(a,n)
= \otimes_{i=1}^{n} L(a,i)$.  We show that $L(\frac{1}{f+j},i) \cdot
F_{a,b,c,d} =0$ for each $i=1,\ldots,n$.  

Suppose first that $n$ is odd.  Then the odd entry is
$2-(n-1)\frac{1}{f+j} = \frac{2j}{f+j}$.  
If the odd entry is not on the long tail of $F_{a,b,c,d}$ (that is, $i \not\in a$), then
we have
\[ x_{a} \geq (n-k) \frac{1}{f+j} \geq 1,
\]
where the rightmost inequality holds because $n-k = 2f+1-k$ and we assumed $f+1-k\geq j$.   Then by Lemma
\ref{phi lemma}, we have  $L(\frac{1}{f+j},i) \cdot F_{a,b,c,d} =0$.
If  the odd entry is on the long tail of $F_{a,b,c,d}$, then $x_a$ is
even larger, and once again we have $L(\frac{1}{f+j},i) \cdot
F_{a,b,c,d} =0$.

Similarly, if $n$ is even, then the odd entry is
$\frac{2j-1}{f+j-1}$.  If it is not on the long tail, then 
\[ x_a \geq (n-k)\frac{1}{f-k+1} \geq 1,
\]
and $x_a$ is even larger if the odd entry is on the long tail, 
and hence $L(\frac{1}{f+j},i) \cdot F_{a,b,c,d} =0$.

For the third statement of the lemma: Let $\beta$ be an arbitrary set of weights and let $\pi_{\beta}:
\M_{0,n} \rightarrow \M_{0,\beta}$ be the birational contraction
defined by Hassett.  Let $I \subset \{1,\ldots,n\}$ such that $2 \leq
\#I \leq \lfloor \frac{n}{2} \rfloor$. Then $\pi_{\beta}$ contracts
the divisor $\Delta_{I}$ on $\M_{0,n} $ if and only if $\sum_{i \in I}
b_{i} \leq 1$ and $\#I \geq 3$  or $\sum_{i \in I^{c}} b_{i} \leq 1$
and $\#I^{c} \geq 3$.  

Now let $\beta= (\frac{1}{k},\ldots,\frac{1}{k})$.  
We know that $\Pic(\M_{0,n}) = \operatorname{Span} \{ B_2,\ldots, B_{f} \}$. 
Using the facts of the paragraph above, we see that $\pi_{\beta}$ contracts $B_i$ if $i \geq k$ and $k \geq 3$.  Thus  if $k \geq 3$, we have 
\[  \dim \Pic(\M_{0,\beta})^{S_n} = (f-1) -(k-2) = f-k+1 = \ell.
\]

\end{proof}

\begin{proposition} \label{ampleness} Write $f =
\lfloor \frac{n}{2} \rfloor$.  Let $k \in \{ 2, \ldots, f\}$ and write $\beta = (\frac{1}{k}, \ldots, \frac{1}{k})$ and 
$\ell = f+1-k$ as in the previous section.
%Recall that $\M_{0,(1/k,\ldots,1/k)}$ corresponds to the critical alpha value
%$\alpha = 2/(k+1)$.  
\begin{enumerate}
\item If $n$ is odd, then any combination $D=\sum_{j=1}^{\ell} c_{j} V(\frac{1}{f+j},n)$ 
with all the coefficients $c_{j} \in \Q_{>0}$ is the pullback of an ample
 $\Q$-line bundle on $\M_{0,\beta}$. 
\item If $n$ is even, then any combination $D = \sum_{j=1}^{\ell} c_{j} V(\frac{1}{f+j-1},n)$ 
with all the coefficients $c_{j} \in \Q_{>0}$ is the pullback of an ample
 $\Q$-line bundle on $\M_{0,\beta}$. 
\end{enumerate}

\end{proposition}

\begin{proof}

We prove the first statement of the proposition (when $n$ is odd).
The second statemenet ($n$ even) follows by a similar argument.  

Let $\beta$ be an arbitrary set of weights and let $\pi_{\beta}:
\M_{0,n} \rightarrow \M_{0,\beta}$ be the birational contraction
defined by Hassett.   Let $a \amalg b\amalg c \amalg d$ be a
partition of $\{1,\ldots,n\}$ and suppose $\sum_{i \in a} b_{i} \geq
\sum_{i \in b} b_{i}  \geq \sum_{i \in c} b_{i}  \geq \sum_{i \in d}
b_{i} $.  Then $\pi_{\beta}$ contracts the F-curve $F_{a,b,c,d}$ if
and only if $ \sum_{i \in b \cup c \cup d} b_{i} \leq 1$.   Moreover,
$\pi_{\beta}$ is a composition of extremal contractions (in fact,
smooth blowdowns) corresponding to the images of classes of F-curves.

Now let $\beta= (\frac{1}{k},\ldots,\frac{1}{k})$ as in the
proposition.  Each $L(\frac{1}{f+j},i)$ is the pullback of a $\Q$-divisor 
on $\M_{0,\beta}$, since $L(\frac{1}{f+j},i)$ is semiample and
$L(\frac{1}{f+j},i) \cdot F_{a,b,c,d} =0$ for every F-curve
$F_{a,b,c,d}$ which is contracted by $\pi_{\beta}$, using the previous
lemma.  Moreover, we can
even say that each $L(\frac{1}{f+j},i)$ is the pullback of a nef
$\Q$-divisor on $\M_{0,\beta}$, since $\pi_{\beta}$ is
surjective.  Symmetrizing, we have: each line bundle
$V(\frac{1}{f+j},n)$ is the pullback of a symmetric $\Q$-divisor $W_{j}$ on
$\M_{0,\beta}$.  Let $D' :=\sum c_j W_j$.   Then $\pi_{\beta}^{*} D' =
D$.  

We claim that $D'$ is ample on $\M_{0,\beta}$.  It is enough to show
that $D'$ is ample on $\M_{0,\beta} / S_n$, since the quotient
morphism $\M_{0,\beta} \rightarrow \M_{0,\beta} / S_n$ is finite.  We
know that $\dim \Pic(\M_{0,\beta})^{S_n} = \ell$ by the previous
lemma.  Since the set $\{ V(\frac{1}{f+j},n)\}_{j=1}^{\ell}$ is
linearly independent, so are the $W_{j}$.  
Since $D' :=\sum c_j W_j$,
and we assumed that all the coefficients $c_j$ are positive, $D'$ is in the
interior of the cone $\langle W_j \mid 1 \leq j \leq \ell \rangle$.
This is a full-dimensional subcone of the nef cone of
$\M_{0,\beta}/S_n$, which is a full-dimensional subcone of
$\Pic(\M_{0,\beta})^{S_n}$.  Hence, by Kleiman's criterion, $D'$ is
ample on $\M_{0,\beta}/S_n$.

The case where $n$ is even can be established by a similar argument.
\end{proof}

\begin{corollary} \label{ample on M_0,beta} Write $f =
\lfloor \frac{n}{2} \rfloor$.  Let $k \in \{ 2, \ldots, f\}$, let 
$\beta = (\frac{1}{k}, \ldots, \frac{1}{k})$,  and  let 
$\alpha = \frac{2}{k+1}$.  Then $A_{\alpha}$ is the pullback
of an ample $\Q$-line bundle on
$\M_{0,\beta}$.
\end{corollary}
\begin{proof}
First suppose $n$ is odd. By studying the coefficients $c_i$ in
Proposition \ref{exp comb} carefully, we see that $A_{\alpha}$ is
a strictly 
positive combination of $V(\frac{1}{\lfloor n/2 \rfloor +1},n)$, \ldots,
$V(\frac{1}{n-k},n)$.   Hence we may use Proposition
\ref{ampleness} to obtain the desired result.

Now suppose $n$ is even.  By studying the coefficients $c_i$ in
Proposition \ref{n even exp comb}, 
we see that $A_{\alpha}$ is a strictly
positive combination of $V(\frac{1}{n/2 },n)$, \ldots,
$V(\frac{1}{n-k-1},n)$.   We apply Proposition
\ref{ampleness} to obtain the desired result.
\end{proof}

Corollary \ref{ample on M_0,beta} together with \cite{SimpA} Corollary 2.3.5
now yields the following result, which has also recently been proved
by Fedorchuk and Smyth, and also by Kiem and Moon.

\begin{corollary}[extending \cite{SimpA} Theorem 2.4.5; cf. \cite{FS};
  cf. \cite{KM10}]
  Fix $n \geq 4$ and $\alpha$ a rational number in
  $(\frac{2}{n-1},1]$.  Let $\M_{0,n}(\alpha)$ denote the log
  canonical model of $\M_{0,n}$ with respect to $K+\alpha \Delta$.

If $\alpha$ is in the range $(\frac{2}{k+2}, \frac{2}{k+1}]$ for some
$k=1,\ldots,\lfloor \frac{n-1}{2} \rfloor$, then $\M_{0,n}(\alpha)
\cong \M_{0,(1/k,\ldots,1/k)}$.
If $\alpha$ is in the range $(\frac{2}{n-1}, \frac{2}{\lfloor
\frac{n}{2} \rfloor +1}]$, then $\M_{0,n}(\alpha) \cong (\Pro^1)^n
\dblq SL(2)$, where the linearization is given by symmetric weights.
\end{corollary}

\section{Further results}

\subsection{The matrix $U^{-1}$}
Here we give closed formulas for the matrix $U^{-1}$.  This gives an easy test for whether a symmetric divisor $L$ is in the simplicial subcone $\langle V(a,n) \mid 
  a=\dfrac1{t}, \ t\in \mathbb Z,\ \lfloor\dfrac{n}{2}\rfloor \le t\le n-2   \rangle$:  One need only multiply $U^{-1}$ with the vector of intersection numbers $[L] \cdot [C_{i}]$ for the $F$-curve classes $[C_{i}]$ defined above, and check whether the resulting vector has all nonnegative entries.  Conjecturally, this tests for membership in the $\SGC$ (see Question \ref{SGCconj}).  One need only multiply $U^{-1}$ with the vector of intersection numbers $[L] \cdot [C_{i}]$ for the $F$-curve classes $[C_{i}]$ defined above.  If the resulting vector has all nonnegative entries, then $L$ is in the $\SGC$.  Conjecturally, the converse holds.

\begin{proposition} Write  $f = \lfloor n /2 \rfloor$.  
If $n$ is odd, then 
\begin{equation}
(U^{-1})_{ij} = \left \{ \begin{array}{l}
0 \mbox{ if $j < i$} \\
\rule{0pt}{22pt} \displaystyle \frac{f+i}{2(f+2-i)} \mbox{ if $j=i$} \\
\rule{0pt}{22pt} \displaystyle \frac{(f+i)(j-1-f)}{(f+2-i)(f+1-i)(f-i)} \mbox{ if $j > i$}
\end{array} \right.
\end{equation} 
If $n$ is even, then 
\begin{equation}
(U^{-1})_{ij} = \left \{ \begin{array}{l}
0 \mbox{ if $j < i$} \\
\frac{1}{4} \mbox{ if $(i,j) = (1,1)$}\\
\rule{0pt}{22pt} \displaystyle \frac{f+i-1}{2(f+2-i)} \mbox{ if $j=i \geq 2$} \\
\rule{0pt}{22pt} \displaystyle \frac{j-1-f}{2f(f-1)}  \mbox{ if $i=1, j > i$} \\
\rule{0pt}{22pt} \displaystyle \frac{(f+i-1)(j-1-f)}{(f-i+2)(f-i+1)(f-i)} \mbox{ if $i >1, j >i$}
\end{array} \right.
\end{equation} 
\end{proposition}

\subsection{The divisor classes of the $V(a,n)$}
We have computed the classes of the main line bundles $V(a,n)$ in the $B_{i}$ basis (see \ref{Bibasis}).  

\begin{proposition}  Let $n$ be odd.  Write $f = \lfloor n/2 \rfloor$.  Fix $u \in \{1,\ldots,f-1\}$.  Let $b_k$ be the coefficients such that
$V(\frac{1}{f+u},n) = \sum_{k=2}^{f} b_k B_k$.  Then 
\begin{equation}
b_k = \left\{ \begin{array}{ll} 
\frac{2(k-1)k}{n-1} & \mbox{ if $u+k \leq f+1$,} \\
\frac{2(k-1)k}{n-1} + \frac{2(f-u-k+1)(f-u+2)}{f+u}  & \mbox{ if  $u+k = f+2$,}\\
\frac{2(k-1)k}{n-1} + \frac{2(f-u-k+1)k}{f+u} & \mbox{ if $u+k \geq f+3$.}
\end{array}
\right.
\end{equation}
Let $n$ be even.  Write $f = n/2$.  Fix $u \in \{1,\ldots,f-1\}$.  Let $b_k$ be the coefficients such that
$V(\frac{1}{f+u-1},n) = \sum_{k=2}^{f} b_k B_k$.  Then 
\begin{equation}
b_k = \left\{ \begin{array}{ll} 
\frac{2(k-1)k}{n-1} & \mbox{ if $u+k \leq f+1$,} \\
\frac{2(k-1)k}{n-1} + \frac{2(f-u-k+1)(f-u+2)}{f+u-1}  & \mbox{ if  $u+k = f+2$,}\\
\frac{2(k-1)k}{n-1} + \frac{2(f-u-k+1)k}{f+u-1} & \mbox{ if $u+k \geq f+3$.}
\end{array}
\right.
\end{equation}
\end{proposition}

\subsection{The $\V(a,n)$ do not always give $\M_{0,\beta}$'s} 

The $\V(a,n)$ are big and nef, and the previous subsection shows that
effective combinations of sufficiently many of them are ample on
certain $\M_{0,\beta}$.  Here we present an example to show that a
single $\V(a,n)$ taken by itself may not be ample on any
$\M_{0,\beta}$.

\begin{proposition} Suppose that $S$ is an $S_{n}$-equivariant line
  bundle on $\M_{0,n}$, and that the image of $\M_{0,n}$ under the
  linear system $|S|$ is isomorphic to a weighted moduli space
  $\M_{0,\beta}$.  Then it is also isomorphic to a weighted moduli
  space $\M_{0,\beta'}$, where $\beta'$ is a symmetric set of weights.
\end{proposition}
\begin{proof}
Recall from \cite{Hass} that the space of weights $\beta$ has a chamber decomposition, again given by hyperplanes of the form $\sum_{i \in I, \#I \geq 3} \beta_{i} = 1$.  If two sets of weights lie in the interior of the same chamber, then the resulting moduli spaces are isomorphic.  

First note that if the moduli spaces $\M_{0, \sigma \beta}$ all lie in a single chamber, then $\M_{0,\beta'}$, where $\beta'$ is the average of all the $\sigma \beta$, lies in this chamber too. 

Now suppose that $\M_{0,\beta}$ and $\M_{0,\sigma \beta}$ lie in two
different chambers.  There is no isomorphism between $\M_{0,\beta}$
and $\M_{0, \sigma \beta}$ commuting with $\pi_{\beta}$ and
$\pi_{\sigma \beta}$, so this contradicts $S$ being symmetric.  (To
see that no such isomorphism is possible, suppose that $\beta$
satisfies $\sum_{i \in I} \beta_{i} \leq 1$ while $\sum_{i \in I}
\sigma \beta_{i} >1$, and consider the locus $T$ in $\M_{0,n}$ whose
generic point is a curve with two components: the first component
contains the points labelled by $I$, and these and the point of
attachment are allowed to vary; the second component contains the
points of $I^{c}$, which are fixed along with the point of attachment.
Then $T$ is isomorphic to $\M_{0,|I|+1}$, and has dimension $|I|-2
\geq 1$.  But $\pi_{\beta}(T)$ is a point, while $\pi_{\sigma
  \beta}:\oM_{0,|I|+1} \to \oM_{0,(1,x_i,i\in I)}$ is birational. 

\end{proof}

Now we consider a specific example: the line bundle $\V(\frac{1}{6},8)$ when $n=8$.  We compute the intersections of $\V(\frac{1}{6},8)$ with all the $F$-curves:

\begin{displaymath}
\begin{array}{lcccc} & V(\frac{1}{6},8) \\ 
F_{3,3,1,1} &   \frac{2}{3}  \\
F_{4,2,1,1} &   \frac{2}{3}  \\
F_{5,1,1,1} &    1   \\
F_{3,2,2,1} &    \frac{1}{3}   \\
F_{2,2,2,2} &     0
\end{array}
\end{displaymath}

From this data it is clear that the linear system $V(\frac{1}{6},8)$ does not give any
$\M_{0,\beta}$.  If the image were any $\M_{0,\beta}$, then by the proposition, it would be possible to use a symmetric $\beta$.  
$V(\frac{1}{6},8)$ is zero on $F_{2,2,2,2}$, but $F_{2,2,2,2}$ is not contracted under any of the maps $\pi_{\beta}: \M_{0,n} \rightarrow \M_{0,\beta} $ with $\beta$ symmetric. 

\section*{Concluding remarks} Since our paper first appeared in a
preprint form, there
have been a number of interesting developments.  
We mention two of these below:

In \cite{Fed10}, Fedorchuk proves that Hassett's weighted pointed spaces
$\M_{0,\beta}$ are log canonical models of $(\M_{0,n},D_{\mathcal{A}} )$
  where $D_{\mathcal{A}} = \sum_{i<j} (a_i+a_j ) \Delta_{i,j} + \sum_{|I| \geq 3} \Delta_I$.
The GIT divisors $L_{\x}$ play an important role in his proof.

A second development is an emerging program to identify the images of
the morphisms associated to conformal blocks with more classical constructions in
algebraic geometry.  Vector bundles of conformal blocks first appeared
in conformal field theory in the 1980s, and they received a great
deal of attention in the 1990s, when several proofs of the Verlinde
formula appeared.  These vector bundles
$\mathbb{V}(\mathfrak{g},\ell,\vec{\lambda})$ are specified by three
pieces of data: a simple Lie algebra $\mathfrak{g}$, a nonnegative
integer $\ell$ called the level, and a set of weights $\vec{\lambda}$
for $\mathfrak{g}$.  

After our paper first appeared in preprint form, 
Fakhruddin observed that all of the GIT line bundles $L_{\x}$ are conformal blocks.  More
precisely, Fakhruddin shows that $L_{\x}$ is a multiple of $\det
\mathbb{V}(\mathfrak{sl}_2,\ell,\vec{\lambda})$, where $\ell$ and
$\vec{\lambda}$ are chosen so that  $\x = \frac{1}{\ell+1} \vec{\lambda}$.
Alternatively,  $L_{\x}$ is a multiple of $
\mathbb{V}(\mathfrak{sl}_{\ell+1},1,\vec{\lambda})$, which is a line bundle,
where we choose integers $\ell$ and $y_i$ such that $x_i =
\frac{y_i}{\ell+1}$ and use the weights $\vec{\lambda} =
(\omega_{y_1},\ldots,\omega_{y_n})$ (\cite{Fakh} Theorem 4.5 and
Remark 5.3).  Fakhruddin also proved that vector bundles of conformal
blocks on $\M_{0,n}$ are globally generated (\cite{Fakh} Lemma 2.2)

Some further developments of GIT bundles and conformal block bundles are
contained in \cite{GS10,Gian10}.


\begin{thebibliography}{MMMMM}
\bibitem[Gib03]{Gib}Gibney, A.  ``Numerical criteria for divisors on
$\M_g$ to be ample.''  math.AG/0312072
\end{thebibliography}
\end{document}